\documentclass[12pt]{amsproc}
\usepackage{geometry}
\usepackage{graphicx}
\usepackage{amsthm}
\usepackage{amssymb}

\theoremstyle{plain}
\newtheorem{thm}{Theorem}[section]
\newtheorem{cor}[thm]{Corollary}
\newtheorem{lem}[thm]{Lemma}
\newtheorem{prop}[thm]{Proposition}
\newtheorem{defn}[thm]{Definition}
\newtheorem{exa}[thm]{Example}
\newtheorem{rem}[thm]{Remark}

\begin{document}

\title{On Graded $1$-Absorbing Prime Submodules}

\author{Ahmad \textsc{Ka'abneh}}
\address{Department of Mathematics, University of Jordan, Amman, Jordan}
\email{ahmadkabeya@yahoo.com}

\author{Rashid \textsc{Abu-Dawwas}}
\address{Department of Mathematics, Yarmouk University, Irbid, Jordan}
\email{rrashid@yu.edu.jo}



\begin{abstract}
Let $G$ be a group with identity $e$, $R$ be a commutative $G$-graded ring with unity $1$ and $M$ be a $G$-graded unital $R$-module. In this article, we introduce the concept of graded $1$-absorbing prime submodule. A proper graded $R$-submodule $N$ of $M$ is said to be a graded $1$-absorbing prime $R$-submodule of $M$ if for all non-unit homogeneous elements $x, y$ of $R$ and homogeneous element $m$ of $M$ with $xym\in N$, either $xy\in (N :_{R} M)$ or $m\in N$. We show that the new concept is a generalization of graded prime submodules at the same time it is a special graded $2$-absorbing submodule. Several properties of a graded $1$-absorbing prime submodule have been obtained. We investigate graded $1$-absorbing prime submodules when the components $\left\{M_{g}:g\in G\right\}$ are multiplication $R_{e}$-modules.
\end{abstract}

\maketitle

\section{Introduction}

Throughout this article, $G$ will be a group with identity $e$ and $R$ a commutative ring with nonzero unity $1$. Then $R$ is said to be $G$-graded if $R=\displaystyle\bigoplus_{g\in G} R_{g}$ with $R_{g}R_{h}\subseteq R_{gh}$ for all $g, h\in G$ where $R_{g}$ is an additive subgroup of $R$ for all $g\in G$. The elements of $R_{g}$ are called homogeneous of degree $g$. If $x\in R$, then $x$ can be written uniquely as $\displaystyle\sum_{g\in G}x_{g}$, where $x_{g}$ is the component of $x$ in $R_{g}$. Moreover, $R_e$ is a subring of $R$ and $1\in R_e$ and $h(R)=\displaystyle\bigcup_{g\in G}R_{g}$. Let $I$ be an ideal of a graded ring $R$. Then $I$ is said to be graded ideal if $I=\displaystyle\bigoplus_{g\in G}(I\cap R_{g})$, i.e., for $x\in I$, $x=\displaystyle\sum_{g\in G}x_{g}$ where $x_{g}\in I$ for all $g\in G$. An ideal of a graded ring need not be graded.

Assume that $M$ is a left unital $R$-module. Then $M$ is said to be $G$-graded if $M=\displaystyle\bigoplus_{g\in G}M_{g}$ with $R_{g}M_{h}\subseteq M_{gh}$ for all $g,h\in G$ where $M_{g}$ is an additive subgroup of $M$ for all $g\in G$. The elements of $M_{g}$ are called homogeneous of degree $g$. It is clear that $M_{g}$ is an $R_{e}$-submodule of $M$ for all $g\in G$. If $x\in M$, then $x$ can be written uniquely as $\displaystyle\sum_{g\in G}x_{g}$, where $x_{g}$ is the component of $x$ in $M_{g}$. Moreover, $h(M)=\displaystyle\bigcup_{g\in G}M_{g}$. Let $N$ be an $R$-submodule of a graded $R$-module $M$. Then $N$ is said to be graded $R$-submodule if $N=\displaystyle\bigoplus_{g\in G}(N\cap M_{g})$, i.e., for $x\in N$, $x=\displaystyle\sum_{g\in G}x_{g}$ where $x_{g}\in N$ for all $g\in G$. An $R$-submodule of a graded $R$-module need not be graded.

For more details and terminology, see \cite{Hazart} and \cite{Nastasescue}.

\begin{lem}\label{1}(\cite{Farzalipour}, Lemma 2.1)\label{1.3} Let $R$ be a $G$-graded ring and $M$ be a $G$-graded $R$-module.

\begin{enumerate}

\item If $I$ and $J$ are graded ideals of $R$, then $I+J$ and $I\bigcap J$ are graded ideals of $R$.

\item If $N$ and $K$ are graded $R$-submodules of $M$, then $N+K$ and $N\bigcap K$ are graded $R$-submodules of $M$.

\item If $N$ is a graded $R$-submodule of $M$, $r\in h(R)$, $x\in h(M)$ and $I$ is a graded ideal of $R$, then $Rx$, $IN$ and $rN$ are graded $R$-submodules of $M$. Moreover, $(N:_{R}M)=\left\{r\in R:rM\subseteq N\right\}$ is a graded ideal of $R$.
\end{enumerate}
\end{lem}

Similarly, if $M$ is a graded $R$-module, $N$ a graded $R$-submodule of $M$ and $m\in h(M)$, then $(N:_{R}m)$ is a graded ideal of $R$. Also, in particular, $Ann_{R}(M)=(0_{M}:_{R}M)$ is a graded ideal of $R$.

The concept of graded prime ideals and its generalizations have a significant place in graded commutative algebra since they are used in understanding
the structure of graded rings. Recall that a proper graded ideal $P$ of $R$ is said to be graded prime ideal if $x, y\in h(R)$ such that $xy\in P$ implies $x\in P$ or $y\in P$ (\cite{Refai Hailat Obiedat}). Graded prime ideals have been extended to graded modules in \cite{Atani}. A proper graded $R$-submodule $N$ of $M$ is said to be a graded prime submodule if whenever $x\in h(R)$ and $m\in h(M)$ with $xm\in N$, then $m\in N$ or $x\in (N:_{R}M)$. The notion of graded $2$-absorbing ideal, which is a generalization of graded prime ideal, was introduced in \cite{Zoubi Dawwas Ceken} as the following: a proper graded ideal $P$ of $R$ is said to be graded $2$-absorbing if whenever $x, y, z\in h(R)$ such that $xyz\in P$, then either $xy\in P$ or $xz\in P$ or $yz\in P$. Also, Graded $2$-absorbing ideals have been deeply studied in \cite{Naghani Moghimi}. Graded $2$-absorbing submodules have been introduced and studied in \cite{Zoubi Dawwas}. A proper graded $R$-submodule $N$ of a graded $R$-module $M$ is said to be graded $2$-absorbing if whenever $x, y\in h(R)$ and $m\in h(M)$ such that $xym\in N$, then either $xm\in N$ or $ym\in N$ or $xy\in (N:_{R}M)$. Graded $2$-absorbing submodules have been generalized into graded $n$-absorbing submodules in \cite{Hamoda}. Actually, the concept of graded $2$-absorbing submodule is a generalization of graded prime submodules.

Recently, in \cite{Dawwas Tekir}, a new class of graded ideals has been introduced, which is an intermediate class of graded ideals between graded prime ideals and graded $2$-absorbing ideals. A proper graded ideal $P$ of $R$ is said to be a graded $1$-absorbing prime ideal if for each non-units $x, y, z\in h(R)$ with $xyz\in P$, then either $xy\in P$ or $z\in P$. Clearly, every graded prime ideal of $R$ is graded $1$-absorbing prime and every graded $1$-absorbing prime ideal of $R$ is graded $2$-absorbing. The next example shows that not every graded $2$-absorbing ideal of $R$ is graded $1$-absorbing prime.

\begin{exa}\cite{Dawwas Tekir} Consider $R=\mathbb{Z}+3x\mathbb{Z}[x]$ and $G=\mathbb{Z}$. Then $R$ is $G$-graded by $R_{0}=\mathbb{Z}$, $R_{j}=3\mathbb{Z}x^{j}$ for $j\geq1$ and $R_{j}=\{0\}$ otherwise. Now, $P=3x\mathbb{Z}[x]$ is a graded ideal of $R$, so by Lemma \ref{1}, $P^{2}$ is a graded ideal of $R$. By (\cite{Yassine}, Example 2.2), $P^{2}$ is a $2$-absorbing ideal of $R$, and then it is graded $2$-absorbing. On the other hand, $P^{2}$ is not graded $1$-absorbing prime as $3, 9, 3x^{2}\in h(R)$ such that $(3)(9)(3x^{2})\in P^{2}$ while $(3)(9)\notin P^{2}$ and $3x^{2}\notin P^{2}$.
\end{exa}

Thus we have the following chain:
\begin{center}
\textit{graded prime ideals} $\Rrightarrow$ \textit{graded $1$-absorbing prime ideals} $\Rrightarrow$ \textit{graded $2$-absorbing ideals}.
\end{center}
On the other hand, we have another chain:
\begin{center}
\textit{graded prime submodules} $\Rrightarrow$ \textit{graded $2$-absorbing submodules}.
\end{center}
Thus we realize that there is a missing part in the second chain, which is between graded prime submodules and graded $2$-absorbing submodules. Then we
define the missing part of the chain as graded $1$-absorbing prime submodules. In this article, we are motivated by \cite{Ugurlu} to introduce and study the concept of graded $1$-absorbing prime submodule. A proper graded $R$-submodule $N$ of $M$ is said to be a graded $1$-absorbing prime $R$-submodule of $M$ if for all non-unit elements $x, y\in h(R)$ and $m\in h(M)$ with $xym\in N$, either $xy\in (N :_{R} M)$ or $m\in N$. We show that the new concept is a generalization of graded prime submodules at the same time it is a special graded $2$-absorbing submodule. Several properties of a graded $1$-absorbing prime submodule have been obtained. We investigate graded $1$-absorbing prime submodules when the components $\left\{M_{g}:g\in G\right\}$ are multiplication $R_{e}$-modules.

\section{Graded $1$-Absorbing Prime Submodules}

In this section, we introduce and study the concept of graded $1$-absorbing prime submodules.

\begin{defn} Let $M$ be a graded $R$-module and $N$ be a proper graded $R$-submodule of $M$. Then $N$ is said to be a graded $1$-absorbing prime $R$-submodule of $M$ if for all non-unit elements $x, y\in h(R)$ and $m\in h(M)$ with $xym\in N$, either $xy\in (N :_{R} M)$ or $m\in N$.
\end{defn}

\begin{exa}\label{Example 1}Let $R$ be a graded local ring with graded maximal ideal $X$ satisfies $X^{2}=\{0\}$. If $M$ is a graded $R$-module, then every proper graded $R$-submodule of $M$ is graded $1$-absorbing prime. To prove that, let $N$ be a proper graded $R$-submodule of $M$, $x, y\in h(R)$ be non-units and $m\in h(M)$ such that $xym\in N$. Since $xy\in X^{2}=\{0\}$, $xy\in (N:_{R}M)$, and hence $N$ is a graded $1$-absorbing prime $R$-submodule of $M$.
\end{exa}

\begin{prop}\label{Proposition 1}Let $M$ be a graded $R$-module and $N$ be a graded $R$-submodule of $M$.
\begin{enumerate}
\item If $N$ is a graded prime $R$-submodule of $M$, then $N$ is a graded $1$-absorbing prime $R$-submodule of $M$.
\item If $N$ is a graded $1$-absorbing prime $R$-submodule of $M$, then $N$ is a graded $2$-absorbing $R$-submodule of $M$.
\end{enumerate}
\end{prop}

\begin{proof}
\begin{enumerate}
\item Let $x, y\in h(R)$ be non-units and $m\in h(M)$ such that $xym\in N$. Then $z=xy\in h(R)$ with $zm\in N$. Since $N$ is graded prime, either $z=xy\in (N:_{R}M)$ or $m\in N$. Hence, $N$ is a graded $1$-absorbing prime $R$-submodule of $M$.
\item Let $x, y\in h(R)$ and $m\in h(M)$ such that $xym\in N$. If $x$ is unit, then $ym\in N$. If $y$ is unit, then $xm\in N$. Suppose that $x$ and $y$ are non-units. Since $N$ is graded $1$-absorbing prime, either $xy\in (N:_{R}M)$ or $m\in N$, as needed. Hence, $N$ is a graded $2$-absorbing $R$-submodule of $M$.
\end{enumerate}
\end{proof}

The next example shows that the converse of Proposition \ref{Proposition 1} (1) is not true in general.

\begin{exa}\label{Example 2}Consider $R=\mathbb{Z}_{4}$, $M=\mathbb{Z}_{4}[x]$ and $G=\mathbb{Z}$. Then $R$ is $G$-graded by $R_{0}=\mathbb{Z}_{4}$ and $R_{j}=\{0\}$ otherwise. Also, $M$ is $G$-graded by $M_{j}=\mathbb{Z}_{4}x^{j}$ for $j\geq0$ and $M_{j}=\{0\}$ otherwise. As $x\in h(M)$, $N=\langle x\rangle$ is a graded $R$-submodule of $M$. By Example \ref{Example 1}, $N$ is a graded $1$-absorbing prime $R$-submodule of $M$. On the other hand, $N$ is not graded prime $R$-submodule of $M$.
\end{exa}

The next example shows that the converse of Proposition \ref{Proposition 1} (2) is not true in general.

\begin{exa}\label{Example 3}Consider $R=\mathbb{Z}$, $M=\mathbb{Z}_{30}[i]$ and $G=\mathbb{Z}_{2}$. Then $R$ is $G$-graded by $R_{0}=\mathbb{Z}$ and $R_{1}=\{0\}$. Also, $M$ is $G$ graded by $M_{0}=\mathbb{Z}_{30}$ and $M_{1}=i\mathbb{Z}_{30}$. Then $N=\langle\overline{6}\rangle$ is a graded $2$-absorbing $R$-submodule of $M$. On the other hand, $2\in h(R)$ is non-unit and $\overline{3}\in h(M)$ with $(2)(2)(\overline{3})\in N$, but $(2)(2)\notin (N:_{R}M)$ and $\overline{3}\notin N$. Hence, $N$ is not graded $1$-absorbing prime $R$-submodule of $M$.
\end{exa}

So, we have the following chain:
\begin{center}
\textit{graded prime submodules} $\Rrightarrow$ \textit{graded $1$-absorbing prime submodules} $\Rrightarrow$ \textit{graded $2$-absorbing submodules}.
\end{center}
And the converse of each implication is not true in general.

\begin{prop}\label{Proposition 2 (1)}Let $M$ be a graded $R$-module and $N$ be a graded $R$-submodule of $M$. If $N$ is a graded $1$-absorbing prime $R$-submodule of $M$, then $(N:_{R}M)$ is a graded $1$-absorbing prime ideal of $R$.
\end{prop}

\begin{proof}Let $x, y, z\in h(R)$ be non-units such that $xyz\in (N:_{R}M)$. Assume that $m\in M$. Then $xyzm_{g}\in N$ for all $g\in G$. Since $N$ is graded $1$-absorbing prime, either $xy\in (N:_{R}M)$ or $zm_{g}\in N$ for all $g\in G$. If $zm_{g}\in N$ for all $g\in G$, then $zm=z\left(\displaystyle\sum_{g\in G}m_{g}\right)=\displaystyle\sum_{g\in G}zm_{g}\in N$, which implies that $z\in (N:_{R}M)$. Hence, $(N:_{R}M)$ is a graded $1$-absorbing prime ideal of $R$.
\end{proof}

Similarly, one can prove the following:

\begin{prop}\label{Proposition 2 (2)}Let $M$ be a graded $R$-module and $N$ be a graded $R$-submodule of $M$. If $N$ is a graded $1$-absorbing prime $R$-submodule of $M$, then $(N:_{R}m)$ is a graded $1$-absorbing prime ideal of $R$ for all $m\in h(M)-N$.
\end{prop}

The next example shows that the converse of Proposition \ref{Proposition 2 (1)} is not true in general.

\begin{exa}\label{Example 4}Consider $R=\mathbb{Z}$, $M=\mathbb{Z}\times\mathbb{Z}$ and $G=\mathbb{Z}_{2}$. Then $R$ is $G$-graded by $R_{0}=\mathbb{Z}$ and $R_{1}=\{0\}$. Also, $M$ is $G$-graded by $M_{0}=\mathbb{Z}\times\{0\}$ and $M_{1}=\{0\}\times\mathbb{Z}$. As $(3, 0)\in h(M)$, $N=\langle(3, 0)\rangle$ is a graded $R$-submodule of $M$ such that $(N:_{R}M)=\{0\}$ is a graded prime ideal of $R$, so $(N:_{R}M)$ is a graded $1$-absorbing prime ideal of $R$. On the other hand, $3, 2\in h(R)$ are non-units and $(1, 0)\in h(M)$ with $(3)(2)(1, 0)\in N$, but $(3)(2)\notin (N:_{R}M)$ and $(1, 0)\notin N$. Hence, $N$ is not graded $1$-absorbing prime $R$-submodule of $M$.
\end{exa}

\begin{lem}\label{Lemma 1}Let $M$ be a graded $R$-module and $N$ be a graded $1$-absorbing prime $R$-submodule of $M$. For a graded $R$-submodule $L$ of $M$ and for non-unit elements $x, y\in h(R)$, if $xyL\subseteq N$, then either $xy\in (N :_{R} M)$ or $L\subseteq N$.
\end{lem}

\begin{proof}Assume that $L\nsubseteq N$. Then there is an element $0\neq m\in L-N$, and then there exists $g\in G$ such that $m_{g}\notin N$. Note that, $m_{g}\in L$ as $L$ is a graded $R$-submodule. By assumption, we have $xym_{g}\in N$. Since $N$ is a graded $1$-absorbing prime $R$-submodule, either $xy\in (N :_{R} M)$ or $m_{g}\in N$. The second choice implies a contradiction, we conclude that $xy\in (N :_{R} M)$, as desired.
\end{proof}

\begin{thm}\label{Theorem 1}Let $M$ be a graded $R$-module and $N$ be a proper graded $R$-submodule of $M$. Then $N$ is a graded $1$-absorbing prime $R$-submodule of $M$ if and only if whenever $I$, $J$ are proper graded ideals of $R$ and $L$ is a graded $R$-submodule of $M$ such that $IJL\subseteq N$, then either $IJ\subseteq (N :_{R} M)$ or $L\subseteq N$.
\end{thm}

\begin{proof}Suppose that $N$ is a graded $1$-absorbing prime $R$-submodule of $M$. Assume that $I$, $J$ are proper graded ideals of $R$ and $L$ is a graded $R$-submodule of $M$ such that $IJL\subseteq N$. Suppose that $IJ\nsubseteq(N:_{R}M)$. Then there exist $x\in I$ and $y\in J$ such that $xy\notin (N:_{R}M)$, and then there exist $g, h\in G$ such that $x_{g}y_{h}\notin (N:_{R}M)$. Note that, $x_{g}\in I$ and $y_{h}\in J$ as $I$ and $J$ are graded ideals. Since $I$ and $J$ are proper, $x_{g}$ and $y_{h}$ are non-units. Now, $x_{g}y_{h}L\subseteq N$, so by Lemma \ref{Lemma 1}, $L\subseteq N$, as needed. Conversely, let $x, y\in h(R)$ be non-units and $m\in h(M)$ such that $xym\in N$. Then $I=\langle x\rangle$, $J=\langle y\rangle$ are proper graded ideals of $R$ and $L=\langle m\rangle$ is a graded $R$-submodule of $M$ such that $IJL\subseteq N$. By assumption, either $IJ\subseteq(N:_{R}M)$ or $L\subseteq N$. If $IJ\subseteq(N:_{R}M)$, then $xy\in(N:_{R}M)$. Hence, $N$ is a graded $1$-absorbing prime $R$-submodule of $M$.
\end{proof}

Let $M$ be a $G$-graded $R$-module, $N$ be a graded $R$-submodule of $M$ and $g\in G$ such that $N_{g}\neq M_{g}$. In \cite{Atani}, $N$ is said to be a $g$-prime $R$-submodule of $M$ if whenever $r\in R_{e}$ and $m\in M_{g}$ such that $rm\in N$, then either $r\in (N:_{R}M)$ or $m\in N$. Also, in \cite{Dawwas Tekir}, A graded ideal $P$ of $R$ is said to be a $g$-$1$-absorbing prime ideal of $R$ if $P_{g}\neq R_{g}$ and whenever non-unit elements $x, y, z$ in
$R_{g}$ such that $xyz\in P$, then $xy\in P$ or $z\in P$. We introduce the following definition:

\begin{defn} Let $M$ be a $G$-graded $R$-module, $N$ be a graded $R$-submodule of $M$ and $g\in G$ such that $N_{g}\neq M_{g}$. Then $N$ is said to be a $g$- $1$-absorbing prime $R$-submodule of $M$ if for all non-unit elements $x, y\in R_{e}$ and $m\in M_{g}$ with $xym\in N$, either $xy\in (N :_{R} M)$ or $m\in N$.
\end{defn}

\begin{rem}
Clearly, every $g$-prime $R$-submodule is $g$-$1$-absorbing prime. However, the converse is not true in general; $N=\langle x\rangle$ in Example \ref{Example 2} is a graded $1$-absorbing prime $R$-submodule, so it is a $g$-$1$-absorbing prime $R$-submodule for all $g\in G=\mathbb{Z}$, but if we choose $g=1$, then $N$ will be not graded $g$-prime $R$-submodule. Now, we are going to prove that if $R_{e}$ is not local ring, then a graded $R$-submodule will be an $e$-$1$-absorbing prime if and only if it is an $e$-prime $R$-submodule.
\end{rem}

\begin{lem}\label{2}Let $M$ be a graded $R$-module. If $M$ has an $e$-$1$-absorbing prime $R$-submodule that is not an $e$-prime $R$-submodule, then the sum of every non-unit element of $R_{e}$ and every unit element of $R_{e}$ is a unit element of $R_{e}$.
\end{lem}

\begin{proof}Let $N$ be an $e$-$1$-absorbing prime $R$-submodule of $M$ that is not an $e$-prime $R$-submodule. Then there exist a non-unit $r\in R_{e}$ and $m\in M_{g}$ such that $rm\in N$ but $r\notin (N :_{R} M)$ and $m\notin N$. Choose a non-unit element $a\in R_{e}$. Then we have that $ram\in N$ and $m\notin N$. Since $N$ is $e$-$1$-absorbing prime, $ra\in (N :_{R} M)$. Let $u\in R_{e}$ be a unit element. Assume that $a + u$ is non-unit. Then $r(a + u)m\in N$. As $N$ is $e$-$1$-absorbing prime, $r(a + u)\in (N :_{R} M)$. This means that $ru\in (N :_{R} M)$, i.e., $r\in (N :_{R} M)$, which is a contradiction. Thus, we have $a+u$ is a unit element.
\end{proof}

\begin{thm}\label{Theorem 2}Let $M$ be a graded $R$-module. If $M$ has an $e$-$1$-absorbing prime $R$-submodule that is not an $e$-prime $R$-submodule, then $R_{e}$ is a local ring.
\end{thm}

\begin{proof}By Lemma \ref{2}, the sum of every non-unit element of $R_{e}$ and every unit element of $R_{e}$ is a unit element of $R_{e}$, and then by
(\cite{Badawi Celikel}, Lemma 1), $R_{e}$ is a local ring.
\end{proof}

\begin{cor}\label{Corollary 1}Let $R$ be a graded ring such that $R_{e}$ is not local ring. Suppose that $M$ is a graded $R$-module. Then a graded $R$-submodule $N$ of $M$ is an $e$-$1$-absorbing prime $R$-submodule if and only if $N$ is an $e$-prime $R$-submodule of $M$.
\end{cor}

\begin{prop}\label{Proposition 3 (1)}Let $\left\{N_{k}\right\}_{k\in\Delta}$ be a chain of graded $1$-absorbing prime $R$-submodules of $M$. Then $\displaystyle\bigcap_{k\in\Delta}N_{k}$ is a graded $1$-absorbing prime $R$-submodule of $M$.
\end{prop}

\begin{proof}Let $x, y\in h(R)$ be non-units and $m\in h(M)$ such that $xym\in \displaystyle\bigcap_{k\in\Delta}N_{k}$. Suppose that $m\notin \displaystyle\bigcap_{k\in\Delta}N_{k}$. Then $m\notin N_{i}$ for some $i\in\Delta$. Since $xym\in N_{i}$ and $N_{i}$ is graded $1$-absorbing prime, $xy\in (N_{i}:_{R}M)$. For any $k\in \Delta$, if $N_{i}\subseteq N_{k}$, then $xy\in (N_{k}:_{R}M)$, and then $xy\in \left(\displaystyle\bigcap_{k\in\Delta}N_{k}:_{R}M\right)$. If $N_{k}\subseteq N_{i}$, then $m\notin N_{k}$. Since $xym\in N_{k}$ and $N_{k}$ is graded $1$-absorbing prime, $xy\in (N_{k}:_{R}M)$, and then $xy\in \left(\displaystyle\bigcap_{k\in\Delta}N_{k}:_{R}M\right)$. Hence, $\displaystyle\bigcap_{k\in\Delta}N_{k}$ is a graded $1$-absorbing prime $R$-submodule of $M$.
\end{proof}

\begin{lem}\label{3}Let $M$ be a graded $R$-module. If $M$ is finitely generated, then the union of every chain of proper graded $R$-submodules of $M$ is a proper graded $R$-submodule of $M$.
\end{lem}

\begin{proof}Suppose that $M=\langle x_{1}, x_{2},..., x_{n}\rangle$ for some $x_{1},..., x_{n}\in M$. Let $\left\{N_{k}\right\}_{k\in\Delta}$ be a chain of proper graded $R$-submodules of $M$. If $\displaystyle\bigcup_{k\in \Delta}N_{k}=M$, then $x_{1}, x_{2}\in \displaystyle\bigcup_{k\in \Delta}N_{k}$, and then $x_{1}\in N_{j}$ and $x_{2}\in N_{r}$ for some $j, r\in \Delta$. If $N_{j}\subseteq N_{r}$, then $x_{1}, x_{2}\in N_{r}$. If $N_{r}\subseteq N_{j}$, then $x_{1}, x_{2}\in N_{j}$. In fact, this works for all $x_{i}$'s. So, there exists $i\in \Delta$ such that $N_{i}$ contains all the generators $x_{1},...,x_{n}$ of $M$, and then $N_{i}=M$, which is a contradiction.
\end{proof}

\begin{prop}\label{Proposition 3 (2)}Let $\left\{N_{k}\right\}_{k\in\Delta}$ be a chain of graded $1$-absorbing prime $R$-submodules of $M$. If $M$ is finitely generated, then $\displaystyle\bigcup_{k\in\Delta}N_{k}$ is a graded $1$-absorbing prime $R$-submodule of $M$.
\end{prop}

\begin{proof}By Lemma \ref{3}, $\displaystyle\bigcup_{k\in \Delta}N_{k}$ is a proper graded $R$-submodule of $M$. Let $x, y\in h(R)$ be non-units and $m\in h(M)$ such that $xym\in \displaystyle\bigcup_{k\in \Delta}N_{k}$. Suppose that $m\notin \displaystyle\bigcup_{k\in \Delta}N_{k}$. Then $m\notin N_{k}$ for all $k\in \Delta$. Now, since $xym\in \displaystyle\bigcup_{k\in \Delta}N_{k}$, $xym\in N_{i}$ for some $i\in \Delta$, and since $N_{i}$ is graded $1$-absorbing prime, $xy\in (N_{i}:_{R}M)\subseteq\left(\displaystyle\bigcup_{k\in \Delta}N_{k}:_{R}M\right)$. Hence, $\displaystyle\bigcup_{k\in\Delta}N_{k}$ is a graded $1$-absorbing prime $R$-submodule of $M$.
\end{proof}

Let $M$ and $S$ be two $G$-graded $R$-modules. An $R$-homomorphism $f:M\rightarrow S$ is said to be a graded $R$-homomorphism if $f(M_{g})\subseteq S_{g}$ for all $g\in G$ (\cite{Nastasescue}).

\begin{lem}\label{4}(\cite{Dawwas Bataineh}, Lemma 2.16) Suppose that $f:M\rightarrow S$ is a graded $R$-homomorphism of graded $R$-modules. If $K$ is a graded $R$-submodule of $S$, then $f^{-1}(K)$ is a graded $R$-submodule of $M$.
\end{lem}

\begin{lem}\label{5}(\cite{Atani Saraei}, Lemma 4.8) Suppose that $f:M\rightarrow S$ is a graded $R$-homomorphism of graded $R$-modules. If $L$ is a graded $R$-submodule of $M$, then $f(L)$ is a graded $R$-submodule of $f(M)$.
\end{lem}

\begin{prop}\label{Proposition 4}Suppose that $f:M\rightarrow S$ is a graded $R$-homomorphism of graded $R$-modules.
\begin{enumerate}
\item If $K$ is a graded $1$-absorbing prime $R$-submodule of $S$ and $f^{-1}(K)\neq M$, then $f^{-1}(K)$ is a graded $1$-absorbing prime $R$-submodule of $M$.
\item Assume that $f$ is a graded epimorphism. If $L$ is a graded $1$-absorbing prime $R$-submodule of $M$ containing $Ker(f)$, then $f(L)$ is a graded $1$-absorbing prime $R$-submodule of $S$.
\end{enumerate}
\end{prop}

\begin{proof}
\begin{enumerate}
\item By Lemma \ref{4}, $f^{-1}(K)$ is a graded $R$-submodule of $M$. Let $x, y\in h(R)$ be non-units and $m\in h(M)$ such that $xym\in f^{-1}(K)$. Then $f(m)\in h(S)$ such that $xyf(m)=f(xym)\in K$. Since $K$ is graded $1$-absorbing prime, either $xy\in (K:_{R}S)$ or $f(m)\in K$, and then either $xy\in (f^{-1}(K):_{R}M)$ or $m\in f^{-1}(K)$. Hence, $f^{-1}(K)$ is a graded $1$-absorbing prime $R$-submodule of $M$.
\item By Lemma \ref{5}, $f(L)$ is a graded $R$-submodule of $S$. Let $x, y\in h(R)$ be non-units and $s\in h(S)$ such that $xys\in f(L)$. Since $f$ is graded epimorphism, there exists $m\in h(M)$ such that $s=f(m)$, and then $f(xym)=xyf(m)=xys\in f(L)$, which implies that $f(xym)=f(t)$ for some $t\in L$, and then $xym-t\in Ker(f)\subseteq L$, which gives that $xym\in L$. Since $L$ is graded $1$-absorbing prime, either $xy\in (L:_{R}M)$ or $m\in L$, and then either $xy\in (f(L):_{R}S)$ or $s=f(m)\in f(L)$. Hence, $f(L)$ is a graded $1$-absorbing prime $R$-submodule of $S$.
\end{enumerate}
\end{proof}

If $M$ is a $G$-graded $R$-module and $L$ is a graded $R$-submodule of $M$, then $M/L$ is a $G$-graded $R$-module by $(M/L)_{g}=(M_{g}+L)/L$ for all $g\in G$ (\cite{Nastasescue}).

\begin{cor}\label{Corollary 2}Let $M$ be a graded $R$-module and $L\subseteq N$ be graded $R$-submodules of $M$. If $N$ is a graded $1$-absorbing prime $R$-submodule of $M$, then $N/L$ is a graded $1$-absorbing prime $R$-submodule of $M/L$.
\end{cor}

\begin{proof}Define $f:M\rightarrow M/L$ by $f(x)=x+L$. Then $f$ is a graded epimorphism with $Ker(f)=L\subseteq N$. So, by Proposition \ref{Proposition 4}, $f(N)=N/L$ is a graded $1$-absorbing prime $R$-submodule of $M/L$.
\end{proof}

\begin{defn}Let $M$ be a graded $R$-module, $L$ be a proper graded $R$-submodule of $M$ and $N$ be a graded $1$-absorbing prime $R$-submodule of $M$ with $L\subseteq N$. Then $N$ is said to a minimal graded $1$-absorbing prime $R$-submodule with respect to $L$ if there is no a graded $1$-absorbing prime $R$-submodule $P$ of $M$ such that $L\subseteq P\subset N$.
\end{defn}

\begin{thm}\label{Theorem 3}Let $M$ be a graded $R$-module and $L$ be a proper graded $R$-submodule of $M$. If $N$ is a graded $1$-absorbing prime $R$-submodule of $M$ such that $L\subseteq N$, then there exists a minimal graded $1$-absorbing prime $R$-submodule with respect to $L$ that it is contained in $N$.
\end{thm}

\begin{proof}Let $X$ be the set of all graded $1$-absorbing prime $R$-submodules $N_{i}$ of $M$ such that $L\subseteq N_{i}\subseteq N$. Then $X$ is non-empty as it contains $N$. Consider $(X, \supseteq)$. Let $\left\{N_{k}\right\}_{k\in \Delta}$ be a chain in $X$. Then by Proposition \ref{Proposition 3 (1)}, $\displaystyle\bigcap_{k\in \Delta}N_{k}$ is a graded $1$-absorbing prime $R$-submodule of $M$, and then by Zorn's Lemma, $X$ has a maximal element $K$. So, $K$ is a graded $1$-absorbing prime $R$-submodule of $M$ such that $L\subseteq K\subseteq N$. If $K$ is not minimal graded $1$-absorbing prime $R$-submodule with respect to $L$, then there exists a graded $1$-absorbing prime $R$-submodule $P$ of $M$ such that $L\subseteq P\subseteq K$, and then $P\in X$, which implies that $K\subseteq P$, and hence $P=K$. Thus, $K$ is a minimal graded $1$-absorbing prime $R$-submodule with respect to $L$.
\end{proof}

Let $I$ be a graded ideal of $R$. Then $Grad(I)$ is the intersection of all graded prime ideals of $R$ containing $I$ (\cite{Zoubi Qarqaz}). Similarly, if $N$ is a graded $R$-submodule of $M$, then the graded radical of $N$ is $Grad(N)$ which is the intersection of all graded prime $R$-submodules of $M$ containing $N$. Motivated by this, we have the following definition.

\begin{defn}
\begin{enumerate}
\item Let $I$ be a graded ideal of $R$. Then the $1$-graded radical of $I$ is the intersection of all graded $1$-absorbing prime ideals of $R$ containing $I$, and is denoted by $Grad_{1}(I)$. If $I=R$ or $R$ has no graded $1$-absorbing prime ideals, we define $Grad_{1}(I)=R$.
\item Let $N$ be a graded $R$-submodule of $M$. Then the $1$-graded radical of $N$ is the intersection of all graded $1$-absorbing prime $R$-submodules of $M$ containing $N$, and is denoted by $Grad_{1}(N)$. If $N=M$ or $M$ has no graded $1$-absorbing prime $R$-submodules, we define $Grad_{1}(N)=M$.
\end{enumerate}
\end{defn}

\begin{rem}Since every graded prime ideal is graded $1$-absorbing prime, then $Grad_{1}(I)\subseteq Grad(I)$ for all graded ideal $I$ of $R$. Similarly, $Grad_{1}(N)\subseteq Grad(N)$ for all graded $R$-submodule $N$ of $M$.
\end{rem}

Motivated by (\cite{Refai Zoubi}, Proposition 1.2), we state the following proposition, and the proof is elementary.

\begin{prop}\label{Proposition 5}Let $M$ be a graded $R$-module and $N, K$ be two graded $R$-submodules of $M$. Then the following statements hold:
\begin{enumerate}
\item $N\subseteq Grad_{1}(N)$.
\item $Grad_{1}(Grad_{1}(N))\subseteq Grad_{1}(N)$.
\item $Grad_{1}\left(N\bigcap K\right)\subseteq Grad_{1}(N)\bigcap Grad_{1}(K)$.
\end{enumerate}
\end{prop}

\begin{prop}\label{Proposition 5 (5)}Let $M$ be a graded $R$-module and $N$ be a graded $R$-submodule of $M$. Then $Grad_{1}((N:_{R}M))\subseteq(Grad_{1}(N):_{R}M)$.
\end{prop}

\begin{proof}If $Grad_{1}(N)=M$, then it is done. Suppose that $Grad_{1}(N)\neq M$. Then there exists a graded $1$-absorbing prime $R$-submodule $L$ of $M$ such that $N\subseteq L$, and then by Proposition \ref{Proposition 2 (1)}, $(L:_{R}M)$ is a graded $1$-absorbing prime ideal of $R$ such that $(N:_{R}M)\subseteq(L:_{R}M)$, which implies that $Grad_{1}((N:_{R}M))\subseteq (L:_{R}M)$, and hence $Grad_{1}((N:_{R}M))M\subseteq L$. In fact, in similar way, $Grad_{1}((N:_{R}M))M\subseteq L_{i}$ for every graded $1$-absorbing prime $R$-submodule $L_{i}$ of $M$ such that $N\subseteq L_{i}$. So, $Grad_{1}((N:_{R}M))M\subseteq Grad_{1}(N)$, which means that $Grad_{1}((N:_{R}M))\subseteq(Grad_{1}(N):_{R}M)$.
\end{proof}

\begin{prop}\label{6}Let $R$ be a graded ring and $I$ be a graded ideal of $R$. If $I$ is a graded $1$-absorbing prime ideal of $R$, then $I_{e}$ is a $1$-absorbing prime ideal of $R_{e}$.
\end{prop}

\begin{proof}Let $x,y, z\in R_{e}$ be non-units such that $xyz\in I_{e}$. Then $x, y, z\in h(R)$ such that $xyz\in I$. Since $I$ is graded $1$-absorbing prime, either $xy\in I$ or $z\in I$. If $xy\in I$, then $xy\in R_{e}R_{e}\bigcap I\subseteq R_{e}\bigcap I=I_{e}$. If $z\in I$, then $z\in R_{e}\bigcap I=I_{e}$. Hence, $I_{e}$ is a $1$-absorbing prime ideal of $R_{e}$.
\end{proof}

\begin{prop}\label{7}Let $M$ be a $G$-graded $R$-module and $N$ be a graded $R$-submodule of $M$. If $N$ is a graded $1$-absorbing prime $R$-submodule of $M$, then $N_{g}$ is a $1$-absorbing prime $R_{e}$-submodule of $M_{g}$ for al $g\in G$.
\end{prop}

\begin{proof}Let $g\in G$. Suppose that $x, y\in R_{e}$ be non-units and $m\in M_{g}$ such that $xym\in N_{g}$. Then $x, y\in h(R)$ and $m\in h(M)$ such that $xym\in N$. Since $N$ is graded $1$-absorbing prime, either $xy\in(N:_{R}M)$ or $m\in N$. If $xy\in(N:_{R}M)$, then $xy\in R_{e}$ such that $xyM_{g}\subseteq xyM\subseteq N$, also, $xyM_{g}\subseteq R_{e}M_{g}\subseteq M_{g}$, and hence $xyM_{g}\subseteq N\bigcap M_{g}=N_{g}$, which implies that $xy\in (N_{g}:_{R_{e}}M_{g})$. If $m\in N$, then $m\in N\bigcap M_{g}=N_{g}$. Hence, $N_{g}$ is a $1$-absorbing prime $R_{e}$-submodule of $M_{g}$.
\end{proof}

Let $M$ be a $G$-graded $R$-module and $g\in G$. In the sense of \cite{El-Baset}, $M_{g}$ is said to be a multiplication $R_{e}$-module if whenever $N$ is an $R_{e}$-submodule of $M_{g}$, then $N=IM_{g}$ for some ideal $I$ of $R_{e}$. Note that, since $I\subseteq (N :_{R_{e}} M_{g})$, $N = IM_{g}\subseteq (N :_{R_{e}} M_{g})M_{g}\subseteq N$. So, if $M_{g}$ is a multiplication $R_{e}$-module, then $N = (N :_{R_{e}} M_{g})M_{g}$ for every $R_{e}$-submodule $N$ of $M_{g}$.

\begin{prop}\label{Theorem 5}Let $M$ be a $G$-graded $R$-module and $g\in G$ such that $M_{g}$ is a faithful multiplication $R_{e}$-module. Suppose that $I$ is a $1$-absorbing prime ideal of $R_{e}$. Then whenever $x, y\in R_{e}$ are non-units and $m\in M_{g}$ such $xym\in IM_{g}$, then either $xy\in I$ or $m\in IM_{g}$.
\end{prop}

\begin{proof}Apply (\cite{Ugurlu}, Theorem 5) on the $R_{e}$-module $M_{g}$.
\end{proof}

\begin{thm}\label{Corollary 4}Let $M$ be a $G$-graded $R$-module and $g\in G$ such that $M_{g}$ is a faithful multiplication $R_{e}$-module. If $I$ is a graded $1$-absorbing prime ideal of $R$ and $(IM)_{g}\neq M_{g}$, then $IM$ is a $g$-$1$-absorbing prime $R$-submodule of $M$.
\end{thm}

\begin{proof}By Lemma \ref{1} (3), $IM$ is a graded $R$-submodule of $M$. Since $I$ is a graded $1$-absorbing prime ideal of $R$, by Proposition \ref{6}, $I_{e}$ is a $1$-absorbing prime ideal of $R_{e}$. Let $x, y\in R_{e}$ be non-units and $m\in M_{g}$ such that $xym\in IM$. Then $xym=x_{e}y_{e}m_{g}=(xym)_{g}\in (IM)_{g}=I_{e}M_{g}$. By Proposition \ref{Theorem 5}, either $xy\in I_{e}$ or $m\in I_{e}M_{g}\subseteq IM$. If $xy\in I_{e}$, then $xy\in I$ and then $xyM\subseteq IM$, which implies that $xy\in (IM:_{R}M)$. Hence, $IM$ is a $g$-$1$-absorbing prime $R$-submodule of $M$.
 \end{proof}

\begin{prop}\label{Proposition 8} Let $M$ be a graded $R$-module and $N$ be a graded $1$-absorbing prime $R/Ann_{R}(M)$-submodule of $M$ as $M$ is a graded $R/Ann_{R}(M)$-module. If $U(R/Ann_{R}(M))=\left\{r+Ann_{R}(M):r\in U(R)\right\}$, then $N$ is a graded $1$-absorbing prime $R$-submodule of $M$ as $M$ is a graded $R$-module.
\end{prop}

\begin{proof}Let $x, y\in h(R)$ be non-units and $m\in h(M)$ such that $xym\in N$. Then $x+Ann_{R}(M), y+Ann_{R}(M)\in h(R/Ann_{R}(M))$ are non-units such that $(x + Ann_{R}(M))(y +Ann_{R}(M))m = xym+Ann_{R}(M)m\in N$. Since $N$ is a graded $1$-absorbing prime $R/Ann_{R}(M)$-submodule of $M$, either $xy + Ann_{R}(M)\in (N :_{R/Ann_{R}(M)} M)$ or $m\in N$. If $xy + Ann_{R}(M)\in (N :_{R/Ann_{R}(M)} M)$, then $xyM\subseteq N$, which implies that $xy\in (N :_{R} M)$. Hence, $N$ is a graded $1$-absorbing prime $R$-submodule of $M$ as $M$ is a graded $R$-module.
\end{proof}

\begin{thm}\label{Theorem 6}Let $M$ be a $G$-graded $R$-module and $g\in G$ such that $M_{g}$ is a multiplication $R_{e}$-module and $U(R_{e}/Ann_{R_{e}}(M_{g}))=\left\{r+Ann_{R_{e}}(M_{g}):r\in U(R_{e})\right\}$. Assume that $I$ is a graded $1$-absorbing prime ideal of $R$ containing $Ann_{R}(M)$. Then $(I_{e}/Ann_{R_{e}}(M_{g}))M_{g}$ is a $g$-$1$-absorbing prime $R$-submodule of $M$.
\end{thm}

\begin{proof}Since $M_{g}$ is a multiplication $R_{e}$-module, $M_{g}$ is a faithful multiplication $R_{e}/Ann_{R_{e}}(M_{g})$-module (see \cite{El-Baset}, page 759). Also, since $I$ is a graded $1$-absorbing prime ideal of $R$, $I_{e}$ is a $1$-absorbing prime ideal of $R_{e}$ by Proposition \ref{6}, and then $I_{e}/Ann_{R_{e}}(M_{g})$ is a $1$-absorbing prime ideal of $R_{e}/Ann_{R_{e}}(M_{g})$ by (\cite{Ugurlu}, Proposition 7). So, $(I_{e}/Ann_{R_{e}}(M_{g}))M_{g}$ is a $g$-$1$-absorbing prime $R_{e}/Ann_{R_{e}}(M_{g})$-submodule of $M$ by Theorem \ref{Corollary 4}, which implies that $(I_{e}/Ann_{R_{e}}(M_{g}))M_{g}$ is a $g$-$1$-absorbing prime $R$-submodule of $M$ by (\cite{Ugurlu}, Proposition 8).
\end{proof}

\end{document}